\documentclass[german,11pt,twoside,a4paper]{article}
\usepackage{a4,amsxtra,latexsym,amssymb,ifthen,amsmath,fancyheadings}
\usepackage{epsfig}
\usepackage[latin1]{inputenc}
\usepackage[T1]{fontenc}
\usepackage[matrix,curve]{xy}
\usepackage{graphicx}
\usepackage{color}
\newcommand{\lra}{\longrightarrow}
\newcommand{\Q}{{\mathbb{Q}}}

\newcommand{\F}{{\mathbb{F}}}
\renewcommand{\H}{{\mathbb{H}}}
\renewcommand{\L}{{\mathbb{L}}}
\newcommand{\T}{{\mathbb{T}}}
\renewcommand{\P}{{\mathbb{P}}}
\newcommand{\Z}{{\mathbb{Z}}}
\newcommand{\PP}{{\mathbb{P}^3}}

\newtheorem{theorem}{Theorem}[section]

\newtheorem{proposition}[theorem]{Proposition}
\newtheorem{corollary}[theorem]{Corollary}

\newtheorem{remark}[theorem]{Remark}
\newenvironment{proof}{\noindent{\sl proof:}}{\qed}

\newcommand{\qed}{\hspace*{\fill} QED}

\title{Small Resolutions and Non-Liftable Calabi-Yau threefolds}

\author{S. Cynk \quad\qquad D. van Straten}
\begin{document}
\maketitle

\begin{abstract}
We use properties of small resolutions of the ordinary double point in dimension three
to construct smooth non-liftable Calabi-Yau threefolds. In particular, we
construct a smooth projective Calabi-Yau threefold over $\F_3$ that does not
lift to characteristic zero and a smooth projective Calabi-Yau
threefold over $\F_5$ having an obstructed deformation. We also construct many
examples of smooth Calabi-Yau algebraic spaces over $\F_p$ that do not lift to
algebraic spaces in characteristic zero.

\end{abstract}

\section{Introduction}
M. Hirokado (\cite{Hirokado}) and S. Schr\"oer (\cite{Schroeer}) have constructed examples of projective Calabi-Yau threefolds in characteristic $2$ and $3$ that have no liftings to characteristic zero. The question arises if there exist non-liftable examples in higher characteristic, \cite{Geer}, \cite{Ekedahl}.
In this paper we construct further examples of non-liftable Calabi-Yau threefolds, most notably a rigid Calabi-Yau threefold in characteristic $3$ and a
projective Calabi-Yau threefold in characteristic $5$ with an obstructed
deformation. Using our method, it is easy to produce non-liftable Calabi-Yau
threefolds
in higher characteristic in the category of algebraic spaces. Up to now, we
were unable to find further projective examples.

Our method exploits a remarkable feature of birational geometry in dimension $\ge 3$, namely the appearance of singularities that have a {\em small resolution}, that is,  admit a map $\pi: Y \lra X$ such that the exceptional set has codimension $\ge 2$ in $Y$.
The key example of the ordinary double point in dimension three was described by Atiyah \cite{Atiyah}.
It admits two  different small resolutions that contract a single rational curve with normal bundle
 ${\cal O}(-1) \oplus {\cal O}(-1)$. Such a rational curve is a {stable submanifold}
in the sense that it lifts to {any deformation} of the ambient variety, \cite{Kodaira}.
As a consequence, blowing down deformations of such small
resolutions give deformations of the variety which remain
singular, \cite{Friedman}. This is very different form  the
geometry related to the ordinary double point in dimension two.
We use this peculiar property of the three dimensional ordinary
double point to give a criterion for varieties in characteristic $p$ to have no lifting to characteristic zero. The criterion can be applied to certain rigid Calabi-Yau threefolds that can be obtained as resolution of a double cover of $\P^3$, ramified over a reducible hypersurface of degree eight and fibre products of rational elliptic surfaces.

\section{Generalities on Liftings}
Consider a scheme $f: X \lra T$ flat over $T$ and assume
$T \hookrightarrow S$ embeds $T$ as a closed subscheme in $S$.
By a {\em lifting} of $X$ to $S$ we mean a scheme $F:{\cal X} \lra S$ flat over $S$,
such that $X={\cal X} \times
_S T$.
We are mainly interested in the case $T=Spec(k)$, where $k$ is a
field of characteristic $p>0$ and $S$ is artinian local ring with $k$ as
residue field. Such liftings are usually constructed by a stepwise process from
situations where $T \subset S$ is defined by an ideal $I$ of square zero,
which will be assumed from now on. One could in fact reduce to the case
where $T \hookrightarrow S$ is a {\em simple extension}, that is, if $S$ is defined by an
ideal $I$ that is annihilated by the maximal ideal of $S$.

The liftings of $f:X \lra T$ to $S$ form a category (in fact a groupoid) where
morphisms between two liftings are defined to be an isomorphism over $S$, which
restricts to the identity over $T$.

If $f:X \lra T$ is {\em smooth}, then liftings always exist locally; the local
lifts of $U \subset X$ form a torsor over $\Theta_{U/T}:=Hom_U(\Omega_{U/T},{\cal O}_U)$.
There is an obstruction element $Ob(X/T,S) \in H^2(X,\Theta_{X/T} \otimes f^*I)$ that
vanishes precisely when $f:X \lra T$ lifts to $S$. If $Ob(X/T,S)=0$, then the
set of isomorphism classes of liftings of $X \lra T$ is a torsor over $H^1(X,\Theta_{X/T} \otimes f^*I)$.
The standard reference for this material is \cite{Grothendieck} or  Expos\'e III in SGA1, \cite{SGA1}.

In the more general case where $f: X \lra T$ is not necessarily smooth, one has to use the theory of the
cotangent complex $\L_{X/T}^{\cdot}$, the derived version of $\Omega_{X/T}$,\cite{IllusieCC}.
For an ${\cal O}_X$-module ${\cal F}$ we put $\T^i(X/T,{\cal F}):=Ext^i_X(\L_{X/T}^{\cdot},{\cal F})=\H^i(Hom_X(\L_{X/T}^{\cdot},{\cal F}))$.
Then one has the following fundamental fact:

\begin{theorem}(\cite{IllusieTN}, Theorem 5.31)
Let $T \hookrightarrow S$ be defined by an ideal $I$ of square zero and $f:X \lra T$
a flat scheme over $T$. Then there is an obstruction
\[Ob(X/T,S) \in \T^2(X/T,f^*I)\]
It is zero precisely if $f:X \lra T$ can be lifted to $S$. If $Ob(X/T,S)=0$,
then the
set of liftings of $f: X \lra T$ is a torsor over $\T^1(X/T,f^*I)$.
The group of automorphisms of any lifting is isomorphic to $\T^0(X/T,f^*I)$.
\end{theorem}

We refer to \cite{IllusieCC}, III 2 for details.
Note that in the case of a simple extension the relevant groups are isomorphic to
$\T^i(X/T)$ ($i=0,1,2$), where we put $\T^i(X/T):=\T^i(X/T,{\cal O}_X)$.

\section{Blowing down Liftings}

The phenomenon of simultaneous resolution was first observed for the simple
singularities by Brieskorn \cite{Brieskorn}. An important feature of the situation
is that under mild cohomological conditions deformations or liftings can
be ``blown-down''. We formulate and sketch a proof of a variant of a theorem
formulated in \cite{Wahl}.

\begin{theorem}
Let $\pi: Y \lra X$ be a morphisms of schemes over $k$ and let $S=Spec(A)$, $A$ artinian
with residue field $k$. Assume that
 ${\cal O}_X=\pi_*({\cal O}_Y)$ and $R^1\pi_*({\cal O}_Y)=0$.
Then for every lifting ${\cal Y} \lra S$ of $Y$ there exists a preferred
lifting ${\cal X} \lra S$ making a commutative diagram
\begin{center}
\[
\begin{array}{ccc}
Y& \hookrightarrow & {\cal Y}\\
\downarrow&&\downarrow\\
X&\hookrightarrow&{\cal X}\\
\end{array}
\]
\end{center}
\end{theorem}
\begin{proof}
A lifting of $X$ over $S$ is defined by giving for each
affine open subset $U \subset X$ a flat ${\cal O}_S$-algebra, reducing to ${\cal O}_U$ over $k$. These algebras should be compatible in the sense that isomorphisms are given for the restrictions to overlaps, which satisfy a cocycle condition.
Let $V:=\pi^{-1}(U) \subset Y$. If a lifting ${\cal Y} \lra S$ is given, we
obtain an ${\cal O}_S$-algebra $A(U):=H^0(V,{\cal O}_{\cal Y})$.

The vanishing of $R^1\pi_*({\cal O}_Y)$ implies that $H^1(V,{\cal O}_Y)=0$.
By reducing to the case of a simple extension and using the long exact cohomology sequence, one sees that this implies that $A(U)$ is ${\cal O}_S$-flat, whereas $\pi_*({\cal O}_Y)={\cal O}_X$ shows that $Spec(A(U))$ indeed is a lift of $U$ over $S$. The compatibilities for gluing are implied by the corresponding properties for ${\cal Y}$.
\end{proof}

Next we want to compare the groups $\T^i(Y/T)$ and $\T^i(X/T)$. Recall that these are
hypercohomology groups of the cotangent complex, which can be calculated using a
local-to-global spectral sequence that reads
\[ E_2^{i,j}=H^i(X,{\cal T}^j_{X/T}) \Longrightarrow \T^{i+j}(X/T)\]
Here ${\cal T}^j_{X/T}$ are the cohomology sheaves of the cotangent complex, so
${\cal T}^0_{X/T} = \Theta_{X/T}$. The sheaf ${\cal T}^1_{X/T}$ determines the
local infinitesimal deformations and ${\cal T}^2_{X/T}$ the local obstructions.

\begin{proposition} Let $\pi: Y \lra X$ be a morphism of schemes over $T$.\\
If ${\Theta}_{X/T}=\pi_*({\Theta}_{Y/T})$ and $R^1\pi_*({\Theta}_{Y/T})=0$
then there is an injection
\[ \T^1(Y/T) \hookrightarrow \T^1(X/T)\]
\end{proposition}
\begin{proof}
From the Leray spectral sequence
\[E_2^{i j}=H^i(X,R^j\pi_*(\Theta_{Y/T})) \Longrightarrow H^{i+j}(Y,\Theta_{Y/T})\]
one obtains an isomorphism
$H^0(Y,\Theta_{Y/T}) \simeq H^0(X,\pi_*(\Theta_{Y/T}))$ and
an exact sequence
\[ 0 \lra H^1(X,\pi_*(\Theta_{Y/T})) \lra  H^1(Y,\Theta_{Y/T}) \lra H^0(X, R^1\pi_*(\Theta_{Y/T}))\lra \ldots \]
As by assumption ${\Theta}_{X/T}=\pi_*({\Theta}_{Y/T})$ and $R^1\pi_*({\Theta}_{Y/T})=0$, we
obtain an isomorphism
\[H^1(X,\Theta_{X/T}) \simeq H^1(Y,\Theta_{Y/T})\]

From the local-to-global spectral sequence of deformations one obtains
\[0 \lra H^1(X,\Theta_{X/T}) \lra {\T}^1(X) \lra H^0(X,{\cal T}^1_{X/T}) \lra \ldots \]
and hence $\T^1(Y/T)=H^1(Y,\Theta_{Y/T})=H^1(X,\Theta_{X/T})$ injects into $\T^1(X/T)$.
\end{proof}

\begin{corollary} Let $\pi: Y \lra X$ be a  morphism of schemes over $T$. Assume that:\\

1) $\pi_*({\cal O}_Y)={\cal O}_Y$ and $R^1\pi_*({\cal O}_Y)=0$.\\
2) ${\Theta}_{X/T}=\pi_*({\Theta}_{Y/T})$ and $R^1\pi_*({\Theta}_{Y/T})=0$.\\

Then there is at most one lifting ${\cal Y} \lra S$ of $Y$ which lies over
a given lifting ${\cal X} \lra S$ of $X$.
\end{corollary}

\begin{proof}
One reduces to the case of simple extensions. As the set of liftings
form a torsor over $\T^1$, it is sufficient to show that there is an injection $\T^1(Y/T) \hookrightarrow \T^1(X/T)$, which is the content of the previous proposition.
\end{proof}

\begin{remark}

{\em
1) It seems that in the above corollary the first condition is implied by the second.\\

2) If ${\Theta}_{X/T}=\pi_*({\Theta}_{Y/T})$ and $R^i\pi_*({\Theta}_{Y/T})=0$ for $i=1,2$
then one can combine the Leray spectral sequence for $\pi:Y \lra X$ and
the local-to-global spectral sequences of deformations into a diagram
{\small
\[
\begin{array}{ccccccccccccc}
0 &\lra&H^1(\Theta_{Y/T})&\stackrel{\simeq}{\lra}& \T^1(Y/T)& \lra & 0 & \lra& H^2(\Theta_{Y/T})& \stackrel{\simeq}{\lra} & \T^2(Y/T)& \lra & 0\\
&&\simeq \downarrow&&\downarrow&&&&\simeq \downarrow&&\downarrow&&\\
 0 &\lra&H^1(\Theta_{X/T})&\lra& \T^1(X/T)& \lra & H^0({\cal T}^1_{X/T}) & \lra& H^2(\Theta_{X/T})& \lra & \T^2(X/T)& \lra & \ldots\\
\end{array}
\]
}
which leads to an exact sequence
\[
0\lra \T^1(Y/T) \lra \T^1(X/T) \lra H^0({\cal T}^1_{X/T}) \lra \T^2(Y/T) \lra \T^2(X/T) \lra\ldots
\]
If in addition ${\cal T}^2_{X/T}=0$, then $\dots$ can be replaced by $0$.}
\end{remark}

\section{Liftings of Small Resolutions}

By a {\em small resolution} of a scheme $X$ we mean a smooth modification
$\pi: Y \lra X$ with the property that the exceptional set has codimension $\ge 2$ in $Y$. Only very special singularities do admit a small resolution and in this section we concentrate on the ordinary double point.\\
Let us study first the local situation. We put $X:=Spec(R)$, where $R:=k[[x,y,z,t]]/(f)$,
where $f=xy-zt$.  Let $Z \subset X$ be the closed subscheme defined by the ideal $(x,z) \subset R$ and let
$Y:=B_Z(X)$ be the blow-up of $X$ in $Z$. $Y$ is described as sub-scheme of
$Spec(k[[x,y,z,t]])\times \P^1$ by the equations
\[xy-zt=0,\;\;xu-zv=0,\;\;(u:v) \in \P^1\]
The exceptional locus is a copy of $\P^1$, with normal bundle ${\cal O}(-1) \oplus {\cal O}(-1)$.
In \cite{Friedman} one finds the following basic calculation.

\begin{theorem} (Friedmann)
Let $\pi: Y \lra X$ be a small resolution of the node in dimension three.
Then:\\
1)  ${\cal O}_X=\pi_*({\cal O}_Y)$ and $R^i\pi_*({\cal O}_Y)=0$ for $i\ge 1$.\\
2)  ${\Theta}_{X/k}=\pi_*({\Theta}_{Y/k})$ and $R^i\pi_*({\Theta}_{Y/k})=0$ for $i \ge 1$.\\
\end{theorem}

From this, one deduces $H^1(\Theta_{Y/k})=H^2(\Theta_{Y/k})=0$. This implies that $Y$ lifts
over any $S$ and does so in a {\em unique} manner.
Explicitly one can construct this lifting of $Y$ over $S$ as follows. We put
\[{\cal X}=Spec({\cal R}),\;\; {\cal R}:={\cal O_S}[[x,y,z,t]]/(xy-zt)\]
and let
\[{\cal Y}=B_Z({\cal X}), \;\;Z=V(x,z) \subset {\cal X}\]

Note that
\[\T^1(X/k)=H^0(X,{\cal T}^1_{X/k})=k[[x,y,z,t]]/(J_f,f)=k[[x,y,z,t]]/(x,y,z,t)=k\]
so the lifting of $X$ is not unique, rather there is a one-dimensional space of
choices.

The lifting of $X$ over which there also is a lifting of $Y$ can be characterized
by having a {\em singular section}
$\sigma:S \lra \cal X$
induced by the canonical projection
\[{\cal R} \lra {\cal O}_S,\;\;x,y,z,t \mapsto 0\]
This is to be contrasted with the lifting of $X$ to
\[{\cal X}=Spec({\cal R}),\;\;{\cal R}={\cal O}_S[[x,y,z,t]]/(xy-zt-q)\]
where ${\cal O}_S/(q)=k$.
As a result of this local analysis one finds the following

\begin{proposition}
Let $X$ be a scheme having $\Sigma \subset X$ as set of nodes and
let $\pi: Y \lra X$ be a small resolution.
Let ${\cal Y}$ be any lift of $Y$ over $S$ and let ${\cal X}$ be the
blow down of ${\cal Y}$. Then for each $x \in \Sigma$ there is
a singular section $\sigma: S \lra {\cal X}$ passing through $x$.
\end{proposition}

As a corollary we obtain the following theorem that forms the basis for
our construction of non-liftable Calabi-Yau spaces. Recall that $X/k$
is called {\em rigid} if $\T^1(X/k)=0$.

\begin{theorem}\label{haupt}
Let ${\cal X}$ be a scheme over $S=Spec(A)$, $A$ a complete domain with
residue field $k$ and fraction field $K=Q(A)$.
Assume that:\\
1) The generic fibre $X_{\eta}:={\cal X} \otimes_A K $ is smooth.\\
2) The special fibre $X:={\cal X} \otimes_A k$ is rigid with nodes as singularities.\\

Let $\pi: Y \lra X$ be a small resolution. Then $Y$ does not lift to $S$.
\end{theorem}
\begin{proof}
Let $A_n:=A/{\bf m}^{n+1}$ and put $S_n:=Spec(A_n)$, $X_n:={\cal X}\times_SS_n \lra S_n$,
so that $X=X_0$.
Assume that $Y$ has a lift ${\cal Y}$ over $S$. From it we get a lift $Y_n:={\cal Y} \times_SS_n$
of $Y$. Then the blow-down of $Y_n \lra S_n$ is a lift of $X$, which by the previous theorem comes with
 singular sections through the singularities.
As $X$ is rigid, the lifting is unique, hence must be isomorphic to the given
lifting $X_n \lra S_n$. Hence, for all $n$ the morphism $X_n \lra S_n$ has a singular section. This
contradicts the smoothness of the generic fibre of ${\cal X} \lra S$. Hence $Y$ can not
lift over $S$.
\end{proof}

\begin{corollary} Under the above assumptions, if $A=\Z_p$, then
$Y$ has no lift to characteristic zero.
\end{corollary}
\begin{proof}
Let ${\cal Y} \lra S$ is a lift of $Y$, where $S$ is a complete DVR which is
a finite extension of $\Z_p$. We get by blowing down a family ${\cal X}' \lra S$ and a map ${\cal Y} \lra {\cal X}'$ over $S$.
Pulling back the given family ${\cal X} \lra Spec(\Z_p)$ via the finite map
$Spec(S) \lra Spec(\Z_p)$ we get a family ${\cal X}'' \lra S$. By rigidity of
$X$ the lifting is unique and we conclude that ${\cal X}'$ and ${\cal X}''$ are isomorphic over $S$. We now get a contradiction to previous theorem \ref{haupt}.
\end{proof}

\begin{remark}{\em The theorem can be understood in terms of $\T^i$ as follows.
As $X$ is assumed to have only nodes as singularities, one has ${\cal T}^2_{X/k}=0$ and
we have an exact sequence
\[0 \lra \T^1(Y/k) \lra \T^1(X/k) \lra H^0(X,{\cal T}^1_{X/k}) \lra \T^2(Y/k) \lra \T^2(X/k) \lra 0\]
If furthermore $X/k$ is rigid, we see that $Y/k$ is rigid as well. The sheaf ${\cal T}^1_{X/k}$
is supported on the singularities and each node contributes a copy of $k$ to $H^0(X,{\cal T}^1_{X/k})$.
Any lifting of $X$ that is ``smoothing out'' a node will produce a non-zero element of
$H^0(X,{\cal T}^1_{X/k})$, which in turn gives a non-zero element in $\T^2(Y/k)$, namely the
obstruction to lift $Y$ over that lifting of $X$. That obstruction element maps to zero in
$\T^2(X/k)$, as it should.}
\end{remark}

\section{Projective examples}

We give some examples where the ideas of the previous sections can be applied.

\subsection{A projective Calabi-Yau threefold over $\F_3$ with no lifting.}

Let $D\subset\PP$ be the  arrangement of eight planes given by the following
equation
\[(x-t)(x+t)(y-t)(y+t)(z-t)(z+t)(x+y+z-t)(x+y+z-3t)=0.\]
This is the arrangement no 86$^{a}$ from the paper \cite{CM}, it
consist of six planes of a cube and two additional planes: one
contains three vertices of the cube but no edges, the other contains
exactly one vertex of the cube and is parallel to the first one. In
fact there are two choices for the second plane and we make
a specific choice, as suggested in the picture.

\begin{figure}[htbp]
\centering\epsfbox{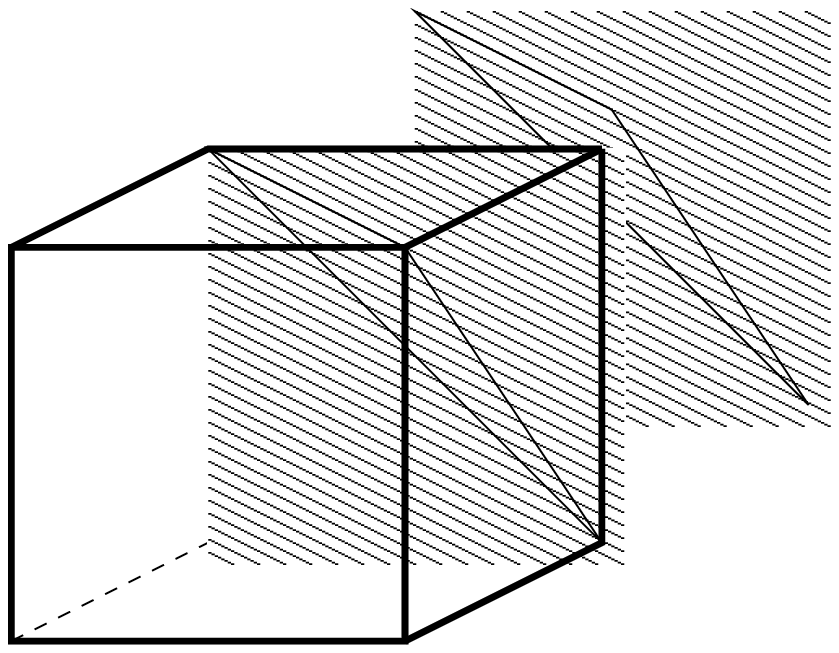}
  \caption{}\label{fig:octic}
\end{figure}

The double cover of $\PP$ branched along this octic surface has a
model ${\cal X}$ over $\Z$, whose generic fibre ${\cal X}_{\eta}$
is a rigid Calabi--Yau manifold with the topological Euler
characteristic $e({\cal X}_{\eta})=80$ and the Hodge numbers
\[ h^{11}=40,\;\;\;h^{12}=0\]
The reduction $X_3$ of ${\cal X}$ modulo 3 turn out to be singular.
To find the singularities of $X_3$ we have
to compare the singularities of the arrangement $D$ and its reduction
$D_{3}$ modulo 3. The only singularities of $D$ are 28 double lines
and 10 fourfold points (of type $p^{0}_{4}$ in the notations introduced
in \cite{CM}), the reduction $D_{3}$ has {\bf one additional
fourfold point}, as in characteristic 3 the plane $x+y+z-3t=0$ contains
the two vertices $(1,1,1,1)$ and $(-1,-1,-1,1)$ of the cube.

To study the resulting singularities, we have to study the resolution process
of the double cover. The general strategy is to blow up in $\PP$ to
make the strict transform of the divisor $D$ non-singular.
This can be done by first blowing--up the 10 fourfold points and
then sequentially the double lines. In characteristic 3 $D$ has an extra
fourfold point, but we still blow--up first the ten fourfold points and
then the double lines. At the eleventh fourfold point, after blowing--up the first
 double line through it, the strict transforms of the two planes that do not contain that line intersect along a {\em cross} (two intersecting lines). When we
blow--up the cross we end up with a threefold with one node and a sum
of eight smooth, pair--wise disjoint surfaces which do not contain the
node.

In local coordinates we can assume that the four planes we are
considering have equations $x=0,y=0,z=0, x+y+z-3=0$. In
characteristic zero we blow--up the double lines. Let us first
blow-up the two disjoint lines $x=y=0$ and $z=x+y+z-3=0$. In one
of the affine charts the blow--up of $\mathbb P^{3}$ is given by
the equation $(v-1)z=x(y+1)-3$. The threefold is smooth
but in characteristic 3 it acquires  a node at
$x=0,y=-1,z=0,v=1$. Since the surface $x=0,z=0$ is  a Weil
divisor on the threefold which is not Cartier  (it is a
component of the exceptional locus of the blow--up) the node
admits a projective small resolution.

The double cover of the threefold branched along the sum of
surfaces, is the reduction $X_{3}$ of ${\cal X}$ modulo 3. It has two
nodes, which form the preimage of the double point under the double cover.
There exists a projective small resolution $X_3$ of the nodes which
is a smooth rigid Calabi--Yau manifold in characteristic 3. The
projective small resolution can be obtained directly from the reduction
$D_{3}$ of the arrangement $D$ modulo 3 by blowing--up first all eleven
fourfold points and then the double lines.

Using the formulas from \cite{CM,CvS} we can compute the Hodge numbers
$h^{11}=42$, $h^{12}=0$, so indeed $Y_3$ is also rigid.
As the nodal $X_3$ is the reduction of ${\cal X}$ with smooth general fibre
${\cal X}_{\eta}$, we can apply the theorem of the previous section to conclude
that
$Y_3$ can not be lifted to characteristic zero.
A local analysis learns that the nodes of $X_3$ are lifted to
\[(\Z/9)[[x,y,z,t]]/(xy-zt-3),\]
which shows that $Y_3$ does not lift to $Spec(\Z/9\Z)$, that is, $Y_3$ has
no lifting at all.

\subsection{A Calabi-Yau threefold over $\F_5$ with an obstructed deformation.}

Consider the  octic arrangement $D\subset\PP$ of eight planes given by (see
fig.~\ref{fig:octic2})
\[ (x-t)(x+t)(y-t)(y+t)(z-t)(z+t)(x+y+Az-At)(x-By-Bz+t)=0.\]

\begin{figure}[htbp]
\centering\epsfbox{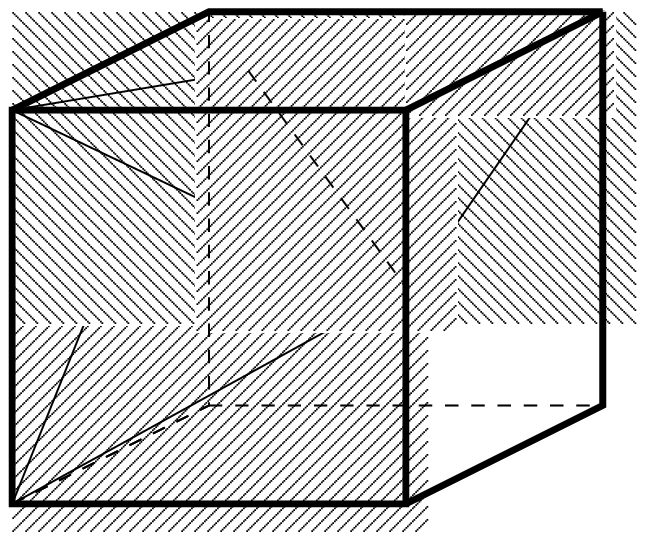}
  \caption{}
\label{fig:octic2}
\end{figure}
For general $A$ and $B$ the arrangement has seven fourfold points
(of type $p_4^0$), namely
the three infinite points $(1,0,0,0)$, $ (0,1,0,0)$, $(0,0,1,0)$ and four
of the vertices of the cube, namely $(1,-1,1,1)$, $(-1,-1,1,1)$ contained
in the plane $(x+y+Az-At)=0$ and $(-1,1,1,1)$, $(-1,1,-1,1)$ contained in the
plane $(x-By-Bz+t)=0$.\\
If there is a special relation between the parameters
$A$ and $B$, the line of intersection of these two planes intersect an edge
of the cube, giving rise to an arrangement with {\em eight} fourfold point.
For example, the two planes meet on the edge $y+t=z+t=0$ precisely when
$A+B=-1$ and on the edge $x-t=y-t=0$ precisely when $AB=A+B$.
In that case the resolution of the double cover branched along $D$
defines a smooth Calabi-Yau manifold with $e=72$
\[h^{11}=37,\;\;\;h^{12}=1\]
as can be computed using \cite{CvS}. Indeed, moving this extra fourfold point along the edge of the cube
produces an equisingular deformation of the arrangement and leads to a
one-parameter family of such Calabi-Yau threefolds.

If {\em both} conditions
\[ A+B=-1,\;\;\;AB=A+B\]
are satisfied, then $A$ and $B$ are roots of the equation
\[x^2+x-1=0\]
and the arrangement has {\em nine} fourfold points.
Hence such an arrangement with nine fourfold points defines
a Calabi-Yau space flat over
\[\Z[x]/(x^2+x-1)\]
The generic fibre is a projective Calabi-Yau threefold with
\[h^{11}=38,\;\;h^{12}=0\]

Now look at the reduction $D_{5}$ of $D$ modulo $5$.
Using \cite{CvS} one computes that one obtains again
a smooth Calabi-Yau manifold $Y_5$, now with
\[h^{11}(X)=39,\;\;h^{12}(X)=1\]
The remarkable thing that happens is
\[x^2+x-1=(x-2)^2 \mod\;5\]
so that
\[\Z[x]/(x^2+x-1) \otimes \F_5 = \F_5[\epsilon], \epsilon^2=0\]

As according to \cite{CvS} the infinitesimal deformations of $Y_5$
are given by equisingular infinitesimal deformations of $D_{5}$.
As the only solution to $x^2+x-1$ in $\mathbb F_{5}$ is 2, in the
ring of dual number $\mathbb F_{5}[\epsilon]$, ($\epsilon^{2}=0$), it
has also the solution $2+\epsilon$, which means that the infinitesimal
equisingular deformations of $D_{5}$ (and hence also infinitesimal
deformations of $X$) are given by arrangement $D_{\epsilon}$ given by
the following equation
\begin{eqnarray*}
 (x-t)(x+t)(y-t)(y+t)(z-t)(z+t)\cdot\rule{5cm}{0cm}\\
\left(x+y+(2+\epsilon)z-(2+\epsilon)t\right)\left(x+(3+\epsilon)y+(3+\epsilon)z+t\right)=0 .
\end{eqnarray*}
This first order deformation cannot be lifted to a second order
deformation (and hence also to a family) since 2 is the only solution
of $x^{2}+x-1=0$ in the fields $\mathbb F_{5}$ and $\bar{\mathbb
   F}_{5}$, whereas in the rings $\mathbb F_{5}[\epsilon]/ \epsilon^{3}$
 and $\bar{\mathbb   F}_{5}[\epsilon]/ \epsilon^{3}$ its solutions are
 $2+c\cdot \epsilon^{2}$ (and similar for higher orders).
The smooth Calabi--Yau manifold $X$ has obstructed
deformations over $\mathbb F_{5}$ and $\bar {\mathbb F}_{5}$.

Similarly one can show that $Y_5$ has no lifting to
$\mathbb Z/25\Z$.

\section{Non-projective Examples}

It is rather easy to produce examples of rigid Calabi-Yau spaces ${\cal X}$
over $\Z$ whose reduction $X={\cal X}$ mod $p$ aquires some extra nodes.
However, in many cases there does not exist a small resolution
$\pi:Y \lra X$ in the categrory of schemes. However, in the larger category
of {\em algebraic spaces} Artin \cite{Artin} has shown that small resolutions do exist.
The arguments in the derivation of theorem \ref{haupt} were of very general nature
and carry over verbatim to the more general context of algebraic spaces.

\subsection{A Calabi-Yau space over $\F_5$ having no lift.}
Consider the arrangement $D\subset \PP$ given by the following equation
\[\left(x^{3}+y^{3}+z^{3}+t^{3}-(x+y+z+t^{3})\right)(x+y)(y+z)(z+t)(x+y+z)(y+z+t)=0.\]

In $\P^4$ this octic surface can be describe by two symmetric equations
\begin{eqnarray*}
&&  (x^{3}+y^{3}+z^{3}+t^{3}+u^{3})(x+y)(y+z)(z+t)(t+u)(u+x)=0\\
 && x+y+z+t+u=0
\end{eqnarray*}
It consists of the Clebsch diagonal cubic  (see fig.~\ref{fig:clebsch})
\[x^{3}+y^{3}+z^{3}+t^{3}-(x+y+z+t)^{3}=0\]

\begin{figure}[htbp]
\centering\epsfbox{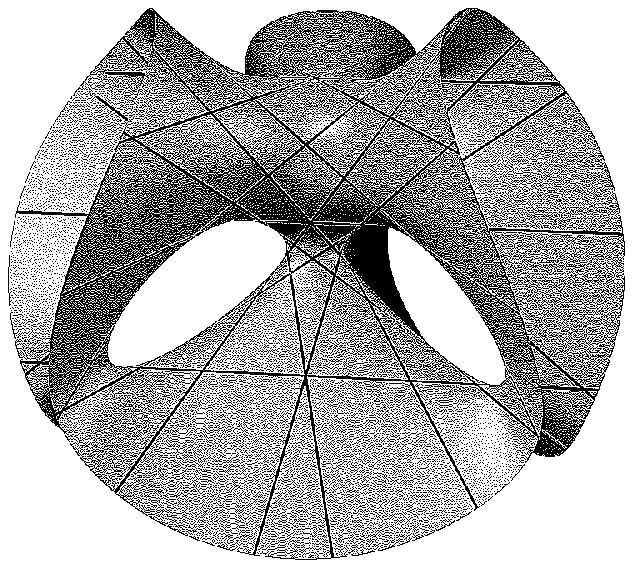}
  \caption{}\label{fig:clebsch}
\end{figure}

and five planes tangent at points where three lines lying on the cubic
meet, these points are called \textbf{Eckardt points}.
A smooth cubic in $\PP$ has at most 10 Eckardt points, the diagonal
Clebsch cubic is the only cubic in $\PP$ that has exactly ten Eckardt
points. One of them is the point $(1,-1,0,0)$, and the other can be
obtained by permutation of coordinates (in $\P^{4}$). Every
Eckardt point is
contained in three line lying on the surface, on any of this lines
there is another Eckardt point. We choose five Eckardt points and five
line lying on the cubic such that the lines form a pentagon with
vertices at the fixed Eckardt points.

The singularities of the surface $D$ consists of ten fourfold points,
ten double and five triple lines.  The fourfold points fall into two
groups
\[
\begin{array}{lcc||ccl}
A_{1}&=&(1,-1,0,0)\rule{12mm}{0cm}&\rule{12mm}{0cm}B_{1}&=&(1,-1,1,-1)\\
A_{2}&=&(0,1,-1,0)\rule{12mm}{0cm}&\rule{12mm}{0cm}B_{2}&=&(0,1,-1,1)\\
A_{3}&=&(0,0,1,-1)\rule{12mm}{0cm}&\rule{12mm}{0cm}B_{3}&=&(-1,0,1,-1\\
A_{4}&=&(0,0,0,1)\rule{12mm}{0cm}&\rule{12mm}{0cm}B_{4}&=&(1,-1,0,1)\\
A_{5}&=&(1,0,0,0)\rule{12mm}{0cm}&\rule{12mm}{0cm}B_{5}&=&(-1,1,-1,0)
\end{array}
\]
The triple lines are
\[\def\arraycolsep{1pt}
\begin{array}{lccccc}
l_{1}:\hspace*{2mm}&x+y&=&z+t&=&0\\
l_{2}:\hspace*{2mm}&x+y&=&z&=&0\\
l_{3}:\hspace*{2mm}&y+z&=&x&=&0\\
l_{4}:\hspace*{2mm}&y+z&=&t&=&0\\
l_{5}:\hspace*{2mm}&z+t&=&y&=&0.
\end{array}
\]
The double lines also fall into two groups
\[
\begin{array}{lccccc||lccccc}
m_{1}:\hspace*{1mm}&x+y&=&t&=&0\hspace*{5mm}&\hspace*{5mm}
n_{1}:\hspace*{1mm}&x+y&=&y+z&=&0\\
m_{2}:\hspace*{1mm}&y+z&=&x+t&=&0\hspace*{5mm}&\hspace*{5mm}
n_{2}:\hspace*{1mm}&x+y&=&y+z+t&=&0\\
m_{3}:\hspace*{1mm}&z+t&=&x&=&0\hspace*{5mm}&\hspace*{5mm}
n_{3}:\hspace*{1mm}&y+z&=&z+t&=&0\\
m_{4}:\hspace*{1mm}&x+z&=&y&=&0\hspace*{5mm}&\hspace*{5mm}
n_{4}:\hspace*{1mm}&z+t&=&x+y+z&=&0\\
m_{5}:\hspace*{1mm}&y+t&=&z&=&0\hspace*{5mm}&\hspace*{5mm}
n_{5}:\hspace*{1mm}&x+y+z&=&y+z+t&=&0\\
\end{array}
\]
The following tables describe incidences of points and lines

\[\def\arraycolsep{8pt}
\begin{array}{c|l| |c|l}
\mbox{line}&\mbox{contains points}&\mbox{point}&\mbox{lies on lines}\\\hline\hline
l_{1}&A_{1},A_{3},B_{2}&A_{1}&l_{1},l_{2},m_{1},n_{4}\\
l_{2}&A_{1},A_{4},B_{1}&A_{2}&l_{3},l_{4},m_{2},n_{5}\\
l_{3}&A_{2},A_{4},B_{2}&A_{3}&l_{1},l_{5},m_{3},n_{2}\\
l_{4}&A_{2},A_{5},B_{5}&A_{4}&l_{2},l_{3},m_{4},n_{1}\\
l_{5}&A_{3},A_{5}B_{3}&A_{5}&l_{4},l_{5},m_{5},n_{3}\\\hline
m_{1}&A_{1},B_{5}&B_{1}&l_{1},m_{2},n_{1},n_{5}\\
m_{2}&A_{2},B_{1}&B_{2}&l_{2},m_{3},n_{3},n_{4}\\
m_{3}&A_{3},B_{2}&B_{3}&l_{3},m_{4},n_{4},n_{5}\\
m_{4}&A_{4},B_{3}&B_{4}&l_{4},m_{5},n_{2},n_{5}\\
m_{5}&A_{5},B_{4}&B_{5}&l_{5},m_{1},n_{1},n_{2}\\\hline
n_{1}&A_{4},B_{1},B_{5}&&\\
n_{2}&A_{3},B_{4},B_{5}&&\\
n_{3}&A_{5},B_{1},B_{2}&&\\
n_{4}&A_{1},B_{2},B_{2}&&\\
n_{5}&A_{2},B_{3},B_{4}&&
\end{array}
\]

Consider the double covering of $\PP$ branched along $D$. We shall
show it has a model ${\cal X}$ over $\Z$, whose generic
fibre ${\cal X}_{\eta}$ is a Calabi--Yau manifold.
Although $D$ is not an octic arrangement in the sense of
\cite{CM}, the singularities are can be resolved in a similar manner.
First we blow--up the fourfold points $A_{i}$ and $B_{i}$.
The lines intersecting at points $B_{i}$'s are now disjoint. Since at the
points $A_{i}$'s one of the planes was threefold tangent to the cubic, after
blowing up that point the strict transform of the cubic and the
tangent plane are still tangent along a line. This line intersects the
other two planes that contained the fourfold point. We have to blow--up
the line, then the strict transforms are transversal so we have to
blow--up a line again.

After this we end up with 10 double and 5 triple lines, these lines are
pairwise disjoint. We blow--up them all, since we add the exceptional
divisors corresponding to the triple lines to the branch locus, we get
15 additional double lines which we have to blow--up again.

Now, we shall compute the Euler characteristic of the resulting
Calabi--Yau manifolds, since we blow--up $\PP$  ten times at a point
and at 40 times at a line --- for every point $A_{i}$ we blow--up twice
at the intersection of the strict transform of cubic at tangent plane,
there are 10 double  and 5 triple lines, but every triple lines is
resolved by blowing--up four times at a line). Consequently
$e(\tilde\PP)=4+10\cdot2+40\cdot2=104$.

The final branch locus is a blow--up  of a sum of a cubic, five planes
and five products $\P^1\times \P^1$. During blow--up of every fourfold
point we blow--up a cubic and three planes, moreover blowing--up every
line of intersection of a strict transform of the cubic and tangent
plane we blow--up two planes. So we get $e(D^{*})=9+5\cdot 3+5\cdot2\cdot2+10\cdot
4+10\cdot2=104$.

The Calabi--Yau manifold $X$ is rigid, since the resolution of the
singularities was obtained by blowing--up points and lines, using
\cite{CvS}, it is enough to show, that every equisingular deformation
of $D$ is trivial. Observe that the points $B_{i}$ lines $L_{i},
m_{i}, n_{i}$ are all determined by the points $A_{i}$. Indeed
intersection of the plane containing three consecutive vertices of
the pentagon formed by points $A_{i}$ with the line containing the
other two vertices gives one of the points $B_{i}$, so we can
determine all the fourfold points, then also the lines are
determined. Since all the choices all five points in general position
in $\PP$ are projectively equivalent, it is enough to check with a
computer algebra program that there
exists exactly one cubic surface which contains all the lines.

So we have $e(X)=104, h^{11}=52, h^{12}=0$. Since during the
resolution we blow--up $\PP$ fifty times, the rank of the subgroup of the Picard group
that is invariant w.r.t the involution is 51. We can
use the standard method of counting points modulo
small primes to show that the Calabi--Yau manifold $X$ is
modular, and
the corresponding modular form has level $720$. If we multiply the
equation of the branch locus by 3, then we get a rigid Calabi--Yau manifold
with the same Hodge numbers and L--series of level five, exactly the
same as the Schoen self product of semistable elliptic fibrations
corresponding to $\Gamma^{0}_{0}(5)$.

The relation to prime $5$ is not surprising, as the Clebsch cubic
is known to be a model of the Hilbert modular surface for $\Q(\sqrt{5})$,
\cite{Hirzebruch}.
The reduction of Clebsch cubic modulo 5 has a node at the point
$(1,1,1,1)$, so $X_{5}$, the reduction mod $5$ of ${\cal X}$, gets
a node that is not present in characteristic zero.
The small resolution of the node $Y_{5}$ has no lifting
to characteristic zero, but turns out to be non-projective.

There is another arrangement with Clebsch cubic that leads to
nonliftable (but non-projective) Calabi--Yau manifold in
characteristic 5.

Consider the octic given in $\mathbb P^{4}$ by the following equations
\begin{eqnarray*}
&&  (x^{3}+y^{3}+z^{3}+t^{3}+u^{3})xyztu=0\\
 && x+y+z+t+u=0
\end{eqnarray*}

The cubic intersects planes in 15 lines with equations
$x+y=z+t=u=0$ and then permuting.
There are another double lines coming from intersection of planes
$x=y=z+t+u=0$. There are ten fourfold points which are intersection of
three lines and the cubic, they have coordinates $(1,-1,0,0,0)$ and
permuted. The resolution comes from blowing--up first fourfold points
and then 25 double lines. The Euler characteristic of smooth model is
84, the Hodge numbers $h^{1,1}=42$, $h^{1,2}=0$. We blow-up 35 times
so there are 6 skew-symmetric divisors, they come from contact planes
with equations $x+y=z+t+u=0$ and symmetric.
In characteristic five there is additional node but again we are not able to find a projective small resolution.

\subsection{Non--liftable Calabi--Yau spaces coming from fiber products.}
In this section we shall consider Calabi--Yau spaces that are
constructed as desingularization of fiber products of
semistable, rational elliptic surfaces with section. This
construction originally goes back to Schoen (\cite{Schoen}).
More precisely we consider twisted self-fiber products of
semistable elliptic surfaces with four singular fibers
(\cite{Schuett}).

There are six types of semi-stable rational elliptic surfaces
with
four singular fibres, the well-known Beauville surfaces, \cite{Beauville}.
Let $S_1$ and $S_2$ be two such surfaces. Assume that the singular fibers corresponds to points
$0,1,\lambda,\infty$ and $0,1,\mu,\infty$, assume that the fibers $(S_1)_\lambda$
and $(S_2)_\mu$ have type $I_1$. Then the fiber product has a (non--projective) small
resolution which is a rigid Calabi--Yau manifold. Assume that $p$ is a prime number such
that $\lambda$ and $\mu$ have the same reduction mod $p$ (different from $0$ and $1$).
Then the reduction $X_p$ of $X$ mod $p$ has an additional node and hence its
small resolution $Y_p$ is a Calabi--Yau space in characteristic $p$ that has
no lifting to characteristic zero.

We list some of the surfaces (the singular fibers and their positions) that we
shall use in our constructions

\def\arraystretch{1.4}
\[\begin{array}{|c|c|c|c|}\hline\hline
\multicolumn{4}{|l|}{\text{Twist of }S(\Gamma_1(5))}\\\hline
\parbox{2cm}{\hfill$0$\hfill\hfill}&\parbox{2cm}{\hfill$1$\hfill\hfill}
&\parbox{2cm}{\hfill$\infty$\hfill\hfill}&\frac12(55\sqrt5-123)\\\hline
I_5&I_1&I_5&I_1\\\hline\hline
\multicolumn{4}{|l|}{\text{Twists of }S(\Gamma_1(6))}\\\hline\hline
0&1&\infty&\frac89\\\hline
I_3&I_6&I_2&I_1\\\hline\hline
0&1&\infty&9\\\hline
I_2&I_3&I_6&I_1\\\hline\hline
0&1&\infty&-8\\\hline
I_3&I_2&I_6&I_1\\\hline\hline
\multicolumn{4}{|l|}{\text{Twists of
}S(\Gamma_0(8)\cap\Gamma_1(4))}\\\hline\hline
\parbox{2cm}{\hfill$0$\hfill\hfill}&\parbox{2cm}{\hfill$1$\hfill\hfill}
&\parbox{2cm}{\hfill$\infty$\hfill\hfill}&
\parbox{2cm}{\hfill$1$\hfill\hfill}
\\\hline
I_2&I_1&I_8&I_1\\\hline\hline
0&1&\infty&2\\\hline
I_1&I_2&I_8&I_1\\\hline\hline
0&1&\infty&\frac12\\\hline
I_8&I_2&I_1&I_1\\\hline\hline
\end{array}
\]
A small resolution of the fiber product of first two surfaces is a
rigid Calabi--Yau manifold $X$ with Euler number
$e(X)=2(5\times3+5\times6+1\times6)=102$ and the Hodge numbers
$h^{11}=51,\,h^{12}=0$. Now,
\begin{eqnarray*}
 &&\dfrac12(55\sqrt5-123)-\frac89=\frac{55\times9\sqrt{5}-(16+123\times9)}{18}
=\\ &&=\dfrac{-4\times9001}{18(55\times9\sqrt{5}+(16+123\times9))}
\end{eqnarray*}
so the reduction $X_{9001}$ of $X$ modulo $9001$ is a rigid Calabi--Yau variety
with one node. Its small resolution $Y_{9001}$ is a smooth Calabi--Yau space
over $\mathbb Z/9001$ that admits no lifting to characteristic 0.

Similarly
\begin{eqnarray*}
 &\dfrac12(55\sqrt5-123)=9\qquad&\text{ in } \quad\mathbb Z/29\quad\text{ and
}\quad\mathbb
Z/41\\
 &\dfrac12(55\sqrt5-123)=-8\qquad&\text{ in } \quad\mathbb Z/919\\
 &\dfrac12(55\sqrt5-123)=-1\qquad&\text{ in } \quad\mathbb Z/11\\
 &\dfrac12(55\sqrt5-123)=2\qquad&\text{ in } \quad\mathbb Z/251\\
 &\dfrac89=9\qquad&\text{ in } \quad\mathbb Z/73\\
&\dfrac89=-8\qquad&\text{ in } \quad\mathbb Z/5\\
&\dfrac89=-1\qquad&\text{ in } \quad\mathbb Z/17\\
&\dfrac89=\dfrac12\qquad&\text{ in } \quad\mathbb Z/7\\
&-1=2\qquad&\text{ in } \quad\mathbb Z/3\\
\end{eqnarray*}
Consequently,we can construct further non--liftable Calabi--Yau spaces in
characteristics $3, 5, 7, 11, 17, 29, 41, 73, 251$ and $919$.

\begin{remark}
The same examples of fiber product Calabi--Yau manifolds were
studied by C.~Schoen (\cite{Schoen2}). He proved that they have
third Betti number equal zero which implies non--liftability.
\end{remark}

\subsection*{Acknowledgments}
The work was done during the first named author's stays at the
Johannes-Gutenberg University in Mainz supported by supported by
DFG Schwerpunktprogramm 1094 (Globale Methoden in der komplexen
Geometrie) and the Sonderforschungsbereich/Transregio 45 (Periods, moduli spaces
and arithmetic of algebraic varieties).
He would like to thank the department for the hospitality and excellent working
conditions.

\vspace{0.8cm}

S\l awomir Cynk:
Instytut Matematyki, Uniwersytetu Jagiello\'nskiego, ul. Reymonta 4,
30-059 Krak\'ow,
POLAND;

{\tt cynk@im.uj.edu.pl}
\bigskip

Duco van Straten: Fachbereich 17,
AG Algebraische Geometrie,
Johannes Gu\-ten\-berg-Universit\"at,
D-55099 Mainz, GERMANY

{\tt straten@mathematik.uni-mainz.de}

\end{document}